\documentclass[12pt]{article}  

\usepackage{amsmath,amsthm,amssymb}
\usepackage[mathscr]{eucal}
\usepackage{enumerate}

\newtheorem{theorem}{Theorem}
\newtheorem{proposition}[theorem]{Proposition}
\newtheorem{definition}[theorem]{Definition}

\newcommand{\Z}{{\mathbb Z}}

\newcommand{\even}{{\bf{e}}}
 \newcommand{\odd}{{\bf{o}}}

\newcommand{\Q}{{\mathbb Q}}

\newcommand{\D}{{\mathcal{D}}_k}

\newcommand{\SL}{{\rm SL}_2 ({\mathbb Z})}

\newcommand{\cz}{\widetilde{\zeta}}

\title{On a conjecture for representations of integers as sums of squares and double shuffle relations}
\author{Koji Tasaka}
\date{}

\begin{document}

\maketitle

\begin{abstract}
In this paper, we prove a conjecture of Chan and Chua for the number of representations of integers as sums of $8s$ integral squares. The proof uses a theorem of Imamo\={g}lu and Kohnen and the double shuffle relations satisfied by the double Eisenstein series of level $2$.
\end{abstract}

\section{Introduction}
For positive integers $s$ and $n$, let $r_s(n)$ (resp. $t_s(n)$) denote the number of representations of $n$ as a sum of $s$ integral squares (resp. triangular numbers). Extensive studies on $r_s(n)$ and $t_s(n)$ have been carried out since the times of Fermat, Euler, and Lagrange. For the histories and recent works of these numbers, we refer the reader to \cite{ck,d,fy,g,kil,ki}. In the current paper, we prove new explicit formulas for $r_{8s}(n)$ and $t_{8s}(n)$ ($s\geq2$) which were conjectured by Chan and Chua \cite{cc} (see also \cite[Remark p.820]{ki}).

For integers $k\ge0$ and $n\ge1$, let $\sigma_{k} (n)$ be defined as usual by $\sigma_{k}(n)=\sum_{d|n} d^{k}$. We define
\[ \sigma_{k}^{i\infty} (n) = \sum_{d|n} (-1)^d d^{k} \mbox{\ and \ } \sigma_{k}^0 (n) =\sum_{\substack{d|n \\ n/d:{\rm odd}}} d^{k} \]
and set $\sigma_{k}^{i\infty} (0) = (1-2^{k+1}) B_{k+1}/2(k+1)$ ($B_k:$ Bernoulli number). We also define 
\[  \rho^{i\infty}_{r,s} (n)= \sum_{m=0}^{n} \sigma_{r}^{i\infty} (m) \sigma_{s}^{i\infty} (n-m) \mbox{\ and\ } \rho^{0}_{r,s} (n)= \sum_{m=1}^{n-1} \sigma_{r}^{0} (m) \sigma_{s}^{0} (n-m) . \]

\begin{theorem}\label{3}
For any positive integer $s\geq2$, there exist unique rational numbers $ \mu_s(l) \ (l=2,3,\ldots,s)$ such that 
\[ r_{8s} (n) = (-1)^n\frac{ 2^{4s}}{(4s-2)!} \sum_{l=2}^s \mu_s(l) {4s-2 \choose 2l-1} \rho^{i\infty}_{4s-2l-1,2l-1} (n) \quad (n\geq0)  \]
and 
\[ t_{8s} (n-s) =\frac{1}{(4s-2)!}  \sum_{l=2}^s \mu_s(l) { 4s-2 \choose 2l-1}  \rho^0_{4s-2l-1,2l-1} (n) \quad (n\geq s). \]
\end{theorem}

\noindent
The first several values of $\mu_s{(l)}$ are given in the following table.

%
%
\begin{tabular}{lllllll}
\hline\noalign{\smallskip}
$\mu_s(l)$ & $l=2$ & $l=3$ & $l=4$ & $l=5$ & $l=6$\\
\noalign{\smallskip}\hline\noalign{\smallskip}
$s=2$ &  $36$  \\
$s=3$ & $420$ & $-200$ \\
$s=4$ & $3168$ & $-3600$ & $1764$  \\
$s=5$ & $21060$ & $-30810$ & $36860$ & $-19116$ \\
$s=6$ &  $\frac{49605048}{343} $& $-\frac{ 77902500}{343}$ & $\frac{15741540}{49}$ & $ -\frac{139785750}{343}$ & $\frac{74727180}{343}$ \\
\noalign{\smallskip}\hline
\end{tabular}

\

For the proof of Theorem \ref{3}, we use the theory of modular forms on the congruence subgroup $\Gamma_0(2)=\left\{\left( \begin{array}{cc} a&b \\ c&d \end{array}\right) \in \SL\ \Big| \ c\equiv 0 \pmod{ 2} \right\}$. Let $\tau$ be a variable on the upper half-plane, and 
\[ \theta (\tau) = \sum_{n\in \Z} q^{n^2} \] 
and 
\[ T(\tau) =q^{1/8} \sum_{n\geq0} q^{n(n+1)/2} \]
be the standard theta functions $(q=e^{2\pi i \tau})$. Then we have, for a positive integer $s$, $\theta(\tau)^s=\sum_{n\geq0} r_s(n)q^n$ and $T(\tau)^{8s} = q^s \sum_{n\geq 0} t_{8s} (n) q^n$, and it is well known that the function $\theta(\tau)^{s}$ is a modular form of weight $s/2$ and level $4$ and the function $T (\tau)^{8s}$ is a modular form of weight $4s$ and level $2$. The two forms $\theta(\tau)^{8s}$ and $T(\tau)^{8s}$ are related to each other by the transformation formula
\[  2^{8s} (2\tau+1)^{-4s} T\left( \frac{-1}{2\tau+1} \right)^{8s} =  \theta (\tau)^{8s}. \]
Let $G_k^0 (\tau)$ and $G_k^{i\infty} (\tau) $ be the Eisenstein series of weight $k$ on $\Gamma_0(2)$ for the cusps $0$ and $i\infty$ (the precise definitions will be recalled in section 2). Then Theorem \ref{3} is equivalent to the following theorem.
\begin{theorem}\label{1} 
For any positive integer $s\geq2$, there exist unique rational numbers $\mu_s(l) \ (l=2,3,\ldots, s)$ such that
\begin{equation}
\label{theta} \theta (\tau)^{8s} =2^{8s} \sum_{l=2}^s\mu_s(l)  G_{2l}^{i\infty} (\tau+\frac{1}{2}) G_{4s-2l}^{i\infty} (\tau+\frac{1}{2}) 
\end{equation}
and
\begin{equation}
\label{tri} T(\tau)^{8s} = \sum_{l=2}^{s} \mu_s(l)  G_{2l}^0 (\tau) G_{4s-2l}^0 (\tau) .
\end{equation}
\end{theorem}

In \cite{ki}, Imamo\={g}lu and Kohnen proved a similar result saying that $T(\tau)^{8s}$ can be expressed as $\Q$-linear combinations of $G_{2l}^0(\tau) G_{4s-2l}^{i\infty}(\tau) \ (l=2,3,\ldots, 2s-2)$ and $G_{4s}^0 (\tau)$. They proved this by showing that the set $\{ G_{2l}^0(\tau) G_{4s-2l}^{i\infty}(\tau) \ (l=2,3,\ldots, 2s-2) \}$ generates the space of cusp forms of weight $4s$ on $\Gamma_0(2)$, using the Rankin--Selberg method and the Eichler--Shimura theory identifying spaces of cusp forms with spaces of periods. We use this result of Imamo\={g}lu and Kohnen and the `double Eisenstein series' to prove the following theorem, from which Theorem \ref{1} follows.

\begin{theorem}\label{2} For each positive even integer $k\geq4$, the set $\{ G^0_k (\tau) ,  G^0_{2l} (\tau) G^0_{k-2l} (\tau) \ ( 2\leq l \leq  [k/4]) \} $ is a basis of the space $ \Q\cdot G_k^0(\tau) \oplus S_k^{\Q} (2)$, where $S_k^{\Q} (2)$ is the $\Q$-vector space spanned by cusp forms of weight $k$ on $\Gamma_0(2)$ with rational Fourier coefficients.
\end{theorem}

In section 2, we will give the proofs of Theorems 1 and 2 assuming Theorem 3. Section 3, which contains a review of the formal double zeta space and the double Eisenstein series, is devoted to the proof of Theorem 3. There we need the double shuffle relations of our newly defined double Eisenstein series whose proof is given in the last section.

\section{Proofs of Theorems 1 and 2}

In this section, we prove Theorems 1 and 2 assuming Theorem 3. For even $k\geq4$, let
\begin{equation}\label{gi}
G_k^{i\infty} (\tau) =(1-2^{k})\frac{B_k}{2k}  + \frac{1}{2^{k} (k-1)!}\sum_{n>0} \sigma_{k-1}^{i\infty} (n) q^n =\frac{1}{2^{k} (k-1)!}\sum_{n\geq 0} \sigma_{k-1}^{i\infty} (n) q^n 
\end{equation}
and
\begin{equation}\label{g0}
G_k^0(\tau) = \frac{1}{(k-1)!} \sum_{n>0} \sigma_{k-1}^0 (n) q^n 
\end{equation}
be the Eisenstein series of weight $k$ on $\Gamma_0(2)$ for the cusps $i\infty$ and $0$ introduced in section 1. Recall that the $\Q$-vector space of modular forms of weight $k$ on $\Gamma_0(2)$ with rational Fourier coefficients is equal to the space $\Q\cdot G_k^0\oplus \Q\cdot G_k^{i\infty} \oplus S_k^{\Q}(2)$. The modular form $T(\tau)^{8s}$ is an element of the space $\Q\cdot G_k^0(\tau) \oplus S_k^{\Q} (2)$ because ${\rm ord}_{\infty} T(\tau)^{8s} >0$, where ${\rm ord}_{\infty} f(\tau) $ is the vanishing order of $f(\tau)$ at $\tau = i \infty \ (q=0)$. Then, assuming Theorem \ref{2}, there exist unique rational numbers $\alpha, \mu_s(l) \ (l=2,3,\ldots, s)$ such that 
\[ T(\tau)^{8s} = \alpha G_{4s}^0(\tau) + \sum_{l=2}^{s} \mu_s(l)  G_{2l}^0 (\tau) G_{4s-2l}^0 (\tau). \]
Since ${\rm ord}_{\infty} T(\tau)^{8s}=s\geq2, {\rm ord}_{\infty} G_{4s}^0 (\tau)=1 $ and ${\rm ord}_{\infty} G_{4s-2l}^0(\tau) G_{2l}^0(\tau) =2$, we find that $\alpha=0$. Thus, we have the assertion \ref{tri} of Theorem 2. On the other hand, from 
\[2^{8s} (2\tau)^{-4s} T ( \frac{-1}{2\tau} )^{8s} = \theta(\tau+\frac{1}{2})^{8s} \]
and 
\begin{align} \label{3a}
(2\tau)^{-k} G_k^0 ( \frac{-1}{2\tau} ) =   G_k^{i\infty} (\tau) \quad (k\geq4 : {\rm even}), 
\end{align}
we have
\begin{align*}
 \theta(\tau+\frac{1}{2})^{8s} &= 2^{8s} (2\tau)^{-4s} T (-\frac{1}{2\tau} )^{8s} \\
&=  2^{8s} (2\tau)^{-4s} \sum_{l=2}^s \mu_s(l) G_{4s-2l}^0 (-\frac{1}{2\tau}) G_{2l}^0 (-\frac{1}{2\tau})\\
&=2^{8s}  \sum_{l=2}^s \mu_s(l) G_{4s-2l}^{i\infty} (\tau+\frac{1}{2}) G_{2l}^{i\infty} (\tau+\frac{1}{2}).
\end{align*}
Hence, we have the formula \eqref{theta} by letting $\tau \rightarrow \tau+1/2$. This completes the proof of Theorem 2. Consequently, we have
\begin{align*}
T(\tau)^{8s} &= q^s\sum_{n\geq0} t_{8s} (n) q^n\\
&= \sum_{n>0} \left( \sum_{l=2}^s \frac{\mu_s(l) }{(4s-2l-1)! (2l-1)!} \sum_{m=1}^{n-1} \sigma_{4s-2l-1}^0(m) \sigma_{2l-1}^0 (n-m) \right) q^n \\
&=  \frac{1}{(4s-2)!} \sum_{n>0} \left( \sum_{l=2}^s \mu_s(l) \binom{4s-2}{2l-1} \rho^0_{4s-2l-1,2l-1} (n) \right) q^n 
\end{align*}
and
\begin{align*}
\theta(\tau)^{8s} &=  \sum_{n\geq0} r_{8s} (n) q^n \\
&= 2^{8s} \sum_{n\geq0}(-1)^n \left( \sum_{l=2}^s \frac{\mu_s(l)}{2^{4s} (4s-2l-1)! (2l-1)!}  \sum_{m=0}^{n} \sigma_{4s-2l-1}^{i\infty} (m) \sigma_{2l-1}^{i\infty} (n-m) \right) q^n \\
&= \frac{2^{4s}}{ (4s-2)!} \sum_{n\geq0} (-1)^n \left( \sum_{l=2}^s \mu_s(l) \binom{4s-2}{2l-1} \rho_{4s-2l-1,2l-1}^{i\infty} (n)  \right) q^n,
\end{align*}
which gives Theorem 1 by comparing the coefficients of $q^n$. \qed

\section{The formal double zeta space and the proof of Theorem \ref{2}}

To prove Theorem \ref{2}, we review some results of \cite{kt}. Throughout this section, we assume $k\geq4$ is even. The formal double zeta space for level $2$, which has been defined by Kaneko and the author \cite{kt}, is the $\Q$-vector space generated by formal variables $Z^{\even \odd}_{r,s},Z^{\odd \even}_{r,s},Z^{\odd \odd}_{r,s},P^{\odd \even}_{r,s},P^{\odd \odd}_{r,s}\ (r+s=k)$ and $Z^{\odd}_k$ with relations
\begin{align}
\label{e_2_1} &\ P^{\odd \even}_{r,s} = Z^{\odd \even}_{r,s} + Z^{\even \odd}_{s,r} = \sum_{i+j=k} \binom{i-1}{r-1} Z^{\odd \even}_{i,j} + \sum_{i+j=k} \binom{i-1}{s-1} Z^{\odd \odd}_{i,j}, \\
\label{e_2_2} & P^{\odd \odd}_{r,s} = Z^{\odd \odd}_{r,s}+ Z^{\odd \odd}_{s,r} +Z^{\odd}_{k} = \sum_{i+j=k} \left( \binom{i-1}{r-1} +\binom{i-1}{s-1} \right) Z^{\even \odd}_{i,j}.
\end{align}
Whenever we write $r+s=k$ or $i+j=k$ without comment,  it is assumed that the variables are integers greater than $1$. We denote this space by $\D$, so that
\begin{align*}
&\D = \frac{ \{ \Q\mbox{-linear combinations of formal variables }Z^{\even \odd}_{r,s},Z^{\odd \even}_{r,s},Z^{\odd \odd}_{r,s},P^{\odd \even}_{r,s},P^{\odd \odd}_{r,s},Z^{\odd}_k  \}}{\langle \mbox{relations\ } \eqref{e_2_1},\eqref{e_2_2} \rangle }.
\end{align*}
For the original version of this space, we refer the reader to \cite{gkz}, which gives the case of $\SL$. In \cite[Theorem 1]{kt}, we proved that every $P^{\odd \even}_{r,s} \ (r,s\geq2: {\rm \ even} )$ with $r+s=k$ can be expressed as $\Q$-linear combinations of $Z^{\odd}_k$ and $P^{\odd \odd}_{r,s} \ (r,s\geq2: {\rm \ even} )$ with $r+s=k$, and that the following summation formula holds:
\begin{equation}\label{e_2_3}
\frac{1}{4} Z^{\odd}_k = \sum_{\begin{subarray}{c} r=2\\ r : {\rm even} \end{subarray}}^{k-2} Z^{\odd \odd}_{r,k-r}.
\end{equation}

\noindent
Therefore, by \eqref{e_2_2} and \eqref{e_2_3}, the $\Q$-vector subspace of $\D$ spanned by $ P^{\odd \even}_{r,s},P^{\odd \odd}_{r,s}\ (r,s:{\rm even})$ and $Z^{\odd}_k$ with $r+s=k$ is generated by $ P^{\odd \odd}_{2r,k-2r} \ ( 2\leq r \leq [k/4])$ and $Z^{\odd}_k$ (we used \eqref{e_2_3} to eliminated $P_{2,k-2}^{\odd\odd}$). Summarizing, for even $k\geq4$, we have
\begin{equation}\label{e_2_4}
\langle P^{\odd \even}_{2r,k-2r},P^{\odd \odd}_{2r,k-2r},Z^{\odd}_k \mid 1\leq r \leq k/2-1 \rangle_{\Q} = \langle  P^{\odd \odd}_{2r,k-2r},Z^{\odd}_k \mid  2\leq r \leq [k/4] \rangle_{\Q}.
\end{equation}

The remainder of this section will be devoted to the proof of Theorem \ref{2}. Let $\even$ (resp. $\odd$) be the set of even integers (resp. odd integers) and $\tau$ be a variable on the upper half-plane. 
Consider the double Eisenstein series, which are defined by
\begin{align}\label{e_2_5} 
G^{ABCD}_{r,s} (\tau)& :=\frac{1}{(2\pi i)^{r+s}} \sum_{\begin{subarray}{c} \lambda > \mu>0 \\ \lambda \in A\tau+B \\ \mu \in C\tau + D \end{subarray}} \frac{1}{\lambda^r \mu^s} =\frac{1}{(2\pi i)^{r+s}} \sum_{\begin{subarray}{c} m\tau+n > m'\tau+n' >0 \\ m\in A,n \in B, m' \in C, n' \in D \end{subarray}} \frac{1}{(m\tau+n)^r (m'\tau+n')^s},
\end{align}
where $A,B,C,D \in \{ \even,\odd ,\Z\}$, and the inequality $m\tau+n>0$ means $m>0$ or $m=0, n>0$ and $m\tau+n>m'\tau+n'$ means $(m-m')\tau+(n-n')>0$. The series \eqref{e_2_5} converge absolutely for $r>2$ and $s>1$. 
Hereinafter, we consider the following three types of double Eisenstein series:
\begin{align} \label{e_2_6} 
Z^{\even \odd}_{r,s} (\tau) &= G_{r,s}^{\even \Z \odd \Z}(\tau) ,\  Z^{\odd \even}_{r,s}(\tau)= G_{r,s}^{\odd \Z \even \Z} (\tau) ,\  Z^{\odd \odd}_{r,s} (\tau) =G_{r,s}^{\odd \Z \odd \Z}(\tau).
\end{align}
For an integer $k\geq1$, we define the Eisenstein series by 
\begin{equation}\label{e_2_7}
G_k(\tau) =\cz(k) +\frac{ (-1)^k}{(k-1)!}\sum_{n>0} \sigma_{k-1} (n)q^n , 
\end{equation}
where $q=e^{2\pi i \tau}$, $\cz(k)=(2\pi i)^{-k}\zeta(k) \ (k\geq2)$, and we set $\cz(1)=0$. We can rewrite \eqref{e_2_7} in the form $G_k(\tau) =(2\pi i)^{-k} \sum_{m\tau+n>0} (m\tau+n)^{-k}$ for $k>2$ by using the Lipschitz formula (see section 4). For $k\geq1$, we let 
\begin{align}\label{e_2_8} 
&G_k^{i\infty} (\tau) = G_k(2\tau) - 2^{-k} G_k(\tau), \\
\label{e_2_9}
&G_k^0 (\tau) = G_k(\tau)-G_k(2\tau) ,
\end{align}
which is compatible with \eqref{gi} and \eqref{g0} when $k\geq4$ is even. We can also rewrite \eqref{e_2_8} and \eqref{e_2_9} in the forms 
\[G_k^{i\infty} (\tau)= \frac{1}{(2\pi i)^k} \sum_{\begin{subarray}{c} m\tau+n>0 \\ m\in \even, n\in \odd \end{subarray} } \frac{1}{(m\tau+n)^k}  , \ G_k^0 (\tau)= \frac{1}{(2\pi i)^k} \sum_{\begin{subarray}{c} m\tau+n>0 \\ m\in \odd, n\in \Z \end{subarray} } \frac{1}{(m\tau+n)^k} \]
for $k>2$.
Note that for odd $k\geq1$, \eqref{e_2_8} and \eqref{e_2_9} define non-zero holomorphic functions on the upper half-plane.  
For positive integers $r$ and $s$, we define 
\[ P^{\odd \even}_{r,s} (\tau) = G_r^0(\tau) G_s(2\tau) +\delta_{r,2} \frac{G_s(2\tau)'}{4s} +\delta_{s,2} \frac{G_r^0(\tau)'}{4r} \]
and
\[  P^{\odd \odd}_{r,s}(\tau) = G_r^0(\tau) G_s^0(\tau) +\delta_{r,2} \frac{G_s^0(\tau)'}{4s} +\delta_{s,2} \frac{G_r^0(\tau)'}{4r} ,\]
where the differential operator $'$ means $q\cdot d/dq$ and $\delta_{r,s}$ is the Kronecker delta.

\noindent
In section 4, we will give the definition of double Eisenstein series \eqref{e_2_6} for all integers $r,s\geq1$ and prove (Theorem \ref{6}) that they satisfy the double shuffle relations (6) and (7) under the identification
\begin{align*}
P_{r,s}^{\odd\even} \leftrightarrow P_{r,s}^{\odd\even} (\tau) ,\ P_{r,s}^{\odd\odd} \leftrightarrow P_{r,s}^{\odd\odd} (\tau) ,\ Z_{k}^{\odd} \leftrightarrow G_k^{0} (\tau) , \ \\
Z_{r,s}^{\even\odd} \leftrightarrow Z_{r,s}^{\even\odd} (\tau) , \ Z_{r,s}^{\odd\even} \leftrightarrow Z_{r,s}^{\odd\even} (\tau) , \ Z_{r,s}^{\odd\odd} \leftrightarrow Z_{r,s}^{\odd\odd} (\tau) .
\end{align*}

Assuming this, we continue the proof of Theorem \ref{2}. An easy computation shows that 
\begin{equation}\label{*}
G_r^{0} (\tau) G_s^{i\infty}(\tau) = 2^{-s} ((2^s-1)P^{\odd \even}_{r,s}(\tau) - P^{\odd \odd}_{r,s}(\tau))\quad (r,s\geq4, \mbox{which implies $\delta_{r,2}=\delta_{s,2}=0$}) . 
\end{equation}
As mentioned in section 1, the products $G_{k-2l}^{0} (\tau) G_{2l}^{i\infty}(\tau) \ (l=2,3,\ldots,k/2-2)$ generate $S_k^{\Q}(2)$. Therefore, by \eqref{e_2_4} and \eqref{*}, the space spanned by $P^{\odd \odd}_{2r,k-2r} (\tau) \ (2\leq r\leq [k/4])$ and $G_k^0(\tau)$ contains $\Q\cdot G_k^0(\tau) \oplus S_k^{\Q}(2)$, which has dimension $[k/4$]. On the other hand, the number of generators of the space $\langle P^{\odd\odd}_{2r,k-2r} (\tau),G_k^0(\tau) \mid  2\leq r\leq [k/4] \rangle_{\Q}$ is equal to $[k/4]$, and hence we have 
\[ \langle P^{\odd \odd}_{2r,k-2r} (\tau),G_k^0(\tau) \mid  2\leq r\leq [k/4] \rangle_{\Q} = \Q\cdot G_k^0(\tau) \oplus S_k^{\Q}(2) .\]
This completes the proof of Theorem 3. \qed

\section{Double Eisenstein series and double shuffle relations}

In this section, we give an extended definition of the double Eisenstein series and prove their double shuffle relations. For $k\geq1$, we can rewrite $G_k$ and $G_k^0$ in the forms 
\begin{align*}
G_k(\tau) =  \cz(k)+\frac{(-1)^k}{(k-1)!} \sum_{\begin{subarray}{c}u>0 \\ m>0\end{subarray}} u^{k-1} q^{um}, \ 
G_k^0(\tau) = \frac{(-1)^k}{(k-1)!} \sum_{\begin{subarray}{c}u>0 \\ m\in\odd_{>0} \end{subarray}} u^{k-1} q^{um}.
\end{align*}
For convenience, we use the following convention:
\[ Z^{\odd}_k(\tau) = G_k^0(\tau) ,\ Z^{\even}_k(\tau) = G_k(2\tau)  .\]
Since the double Eisenstein series defined by \eqref{e_2_6}  is invariant under the translation $\tau \rightarrow \tau+1$, we can consider the Fourier expansion of \eqref{e_2_6} for $s>1$ and $r>2$. To describe the expansion, we use the power series $\varphi_k(\tau)$ in $\Q[[q]]$ defined by
\[ \varphi_k(\tau) = \frac{(-1)^{k}}{(k-1)!} \sum_{u>0} u^{k-1} q^u \quad (k\geq1 ).  \]
\begin{proposition}\label{4} For any integers $s>1$ and $r>2$, we have
\begin{align*}
Z^{\even \odd}_{r,s} (\tau) &= \sum_{\begin{subarray}{c} m>m'>0 \\ m\in\even, m'\in \odd \end{subarray} } \varphi_r (m\tau) \varphi_s (m'\tau), \\
Z^{\odd \even}_{r,s} (\tau) &= \sum_{\begin{subarray}{c} m>m'>0 \\ m\in\odd, m'\in \even \end{subarray} } \varphi_r (m\tau) \varphi_s (m'\tau) + \cz (s) \sum_{m\in \odd_{>0}} \varphi_r(m\tau) , \\
 Z^{\odd \odd}_{r,s} (\tau) &=\sum_{\begin{subarray}{c} m>m'>0 \\ m\in\odd, m'\in \odd \end{subarray} } \varphi_r (m\tau) \varphi_s (m'\tau)\\
 &  + \sum_{p+h=r+s} \left( (-1)^s \binom{p-1}{s-1} + (-1)^{p+r} \binom{p-1}{r-1} \right) \cz(p) \sum_{m\in \odd_{>0} } \varphi_h (m\tau) ,
 \end{align*}
 where we use our usual convention that the condition ``$p+h=r+s$'' tacitly includes ``$p,h\geq1$''.
\end{proposition}
\begin{proof}
We first recall the Lipschitz formula 
\begin{align*}
\lim_{N\rightarrow \infty} \sum_{n=-N}^N \frac{1}{\tau+n} & = -\pi i+(-2\pi i )\sum_{u>0} q^u=-\pi i + 2\pi i \varphi_1(\tau) ,\\
 \sum_{n\in \Z } \frac{1}{(\tau+n)^r} &= \frac{(-2\pi i)^r}{(r-1)!} \sum_{u>0} u^{r-1} q^u =(2\pi i)^r \varphi_r (\tau) \ \ \ (r\geq2).
\end{align*}
We can divide the summation in the defining series \eqref{e_2_5} into four terms, corresponding to $m=m'=0, m>m'=0,m=m'>0$, and $m>m'>0$, where the first term is zero for the cases $Z^{\even \odd}_{r,s},Z^{\odd \even}_{r,s}$, and $Z^{\odd \odd}_{r,s}$ (for convenience, we sometimes drop the variable $\tau$). We only prove the most involved case of $Z^{\odd \odd}_{r,s}$. In this case, we obtain
\[ Z^{\odd \odd}_{r,s} =\frac{1}{(2\pi i)^{r+s}} \Bigl(  \sum_{\begin{subarray}{c}  m=m'>0 \\ n>n' \\ m,m'\in \odd \\ n,n' \in \Z \end{subarray}}  +\sum_{\begin{subarray}{c}  m>m'>0  \\ m,m'\in \odd \\ n,n' \in \Z  \end{subarray}} \Bigr) \frac{1}{(m\tau+n)^r (m'\tau+n')^s} .\]
The second term is easily seen to be 
$$\sum_{\begin{subarray}{c} m>m'>0 \\ m\in\odd, m'\in \odd \end{subarray} } \varphi_r (m\tau) \varphi_s (m'\tau). $$
For the calculation of the first term, we need the partial fraction decomposition 
\begin{align}\label{e_3_3}
\notag \frac{1}{(\tau+n)^r (\tau + n')^s} =& (-1)^s \sum_{i=0}^{r-1} \binom{s+i-1}{i} \frac{1}{(\tau+n)^{r-i}}\cdot \frac{1}{(n-n')^{s+i}} \\
 &+ \sum_{j=0}^{s-1} (-1)^j \binom{r+j-1}{j} \frac{1}{(\tau+n')^{s-j}}\cdot \frac{1}{(n-n')^{r+j}} .
\end{align}
Let $h=n-n'$. Then $h$ is a positive integer. Using \eqref{e_3_3}, the first term can be calculated as
\begin{align*}
&\frac{1}{(2\pi i)^{r+s}} \sum_{\begin{subarray}{c}  m=m'>0 \\ n>n' \\ m,m'\in \odd, n,n' \in \Z \end{subarray}}   \frac{1}{(m\tau+n)^r (m'\tau+n')^s} =\frac{1}{(2\pi i)^{r+s}} \sum_{m\in \odd_{>0}} \sum_{\begin{subarray}{c} n>n' \\ n, n' \in \Z \end{subarray}} \frac{1}{(m\tau+n)^r (m\tau+n')^s} \\
=&\frac{1}{(2\pi i)^{r+s}} \sum_{m\in \odd_{>0} } \sum_{\begin{subarray}{c} n\in \Z \\ h \in \Z_{>0}  \end{subarray}} \left\{ (-1)^s \sum_{i=0}^{r-1} \binom{s+i-1}{i} \frac{1}{(m\tau+n)^{r-i}}  \frac{1}{h^{s+i}} \right. \\
&+ \left. \sum_{j=0}^{s-1} (-1)^j \binom{r+j-1}{j}  \frac{1}{(m\tau +n-h)^{s-j}}   \frac{1}{h^{r+j}} \right\} \\
= & (-1)^s \sum_{i=0}^{r-1} \binom{s+i-1}{i}  \sum_{h\in \Z_{>0}} \frac{(2\pi i)^{-s-i}}{h^{s+i}}  \sum_{m\in \odd_{>0}} \sum_{n \in \Z} \frac{(2\pi i)^{-r+i}}{(m\tau+n)^{r-i}}  \\
& + \sum_{j=0}^{s-1} (-1)^j \binom{r+j-1}{j} \sum_{h\in \Z_{>0}} \frac{(2\pi j)^{-r-j}}{h^{r+j}} \sum_{m \in \odd_{>0}} \sum_{n\in \Z} \frac{(2\pi j)^{-s+j}}{(m\tau +n-h)^{s-j}} \\
= & (-1)^s \sum_{i=0}^{r-2} \binom{s+i-1}{i}  \cz (s+i) \sum_{m\in \odd_{>0}} \varphi_{r-i} (m \tau)\\
& + \sum_{j=0}^{s-2} (-1)^j \binom{r+j-1}{j} \cz (r+j) \sum_{m\in \odd_{>0}} \varphi_{s-j} (m\tau) \\
=& \sum_{p+h=r+s} \left\{ (-1)^s \binom{p-1}{s-1}  + (-1)^{p+r} \binom{p-1}{r-1}\right\} \cz(p)  \sum_{m\in \odd_{>0} } \varphi_h (m\tau)  .
\end{align*}
The cancellation of the terms for $i=r-1$ and $j=s-1$ in the third equality can be justified by computing Cauchy principal values. The final equality is obtained by setting $ s+i=p,r-i=h$ in the first term and $r+j=p,s-j=h$ in the second. This completes the proof for $Z^{\odd \odd}_{r,s}$, the verification of the other cases being left to the reader. 
\end{proof}

\noindent
For a positive integer $r$ and non-negative integer $s$, we define 
\[ f^{\odd}_r (\tau) = \sum_{m\in \odd_{>0}} \varphi_r (m\tau) ,\ f^{\even}_r (\tau) = \sum_{m\in \even_{>0}} \varphi_r (m\tau) \]
and
\[ \overline{f}^{\odd}_s (\tau) = - \sum_{m\in \odd_{>0}} m \varphi_{s+1} (m\tau) , \ \overline{f}^{\even}_s (\tau) = - \sum_{m\in \even_{>0}} m \varphi_{s+1} (m\tau) .\]
We note that for any positive integer $k$, we have 
$$Z^{\odd}_k(\tau)=f^{\odd}_k (\tau),Z^{\even}_k(\tau)= \cz (k) + f^{\even}_k(\tau), Z^{\odd}_k(\tau)'=k \overline{f}^{\odd}_k (\tau),Z^{\even}_k(\tau)'=k \overline{f}^{\even}_k (\tau). $$

To extend the definition of the double Eisenstein series $Z^{\even \odd}_{r,s},Z^{\odd \even}_{r,s}$, and $Z^{\odd \odd}_{r,s}$ to any positive integers $r$ and $s$, we need the correction terms $\varepsilon^{\even \odd}_{r,s} (\tau) ,\varepsilon^{\odd \even}_{r,s} (\tau) $ and $\varepsilon^{\odd \odd}_{r,s} (\tau) $. This is necessary because we require the extended double Eisenstein series to satisfy the double shuffle relations (see \cite{gkz,kt}). We set 
\begin{align*}
\varepsilon^{\even \odd}_{r,s} (\tau) &= \delta_{r,2} \overline{f}^{\odd}_s(\tau) - \delta_{r,1}  \overline{f}^{\odd}_{s-1}(\tau) + \delta_{s,1}  \overline{f}^{\even}_{r-1} (\tau) + \delta_{r,1}\delta_{s,1} \alpha_1  ,\\
\varepsilon^{\odd \even}_{r,s} (\tau) &=   \delta_{r,2} \overline{f}^{\even}_s(\tau) - \delta_{r,1}  \overline{f}^{\even}_{s-1}(\tau) + \delta_{s,1}  \overline{f}^{\odd}_{r-1} (\tau) + \delta_{r,1}\delta_{s,1} \alpha_2   ,\\
\varepsilon^{\odd \odd}_{r,s} (\tau) &=   \delta_{r,2} \overline{f}^{\odd}_s(\tau) - \delta_{r,1}  \overline{f}^{\odd}_{s-1}(\tau) + \delta_{s,1} (  \overline{f}^{\odd}_{r-1} (\tau) +2  f^{\odd}_{r} (\tau) )+ \delta_{r,1}\delta_{s,1} \alpha_3,
\end{align*}
where $\alpha_1=-\alpha_2=\overline{f}_0^{\odd} (\tau)$ and $ \alpha_3=2\overline{f}^{\odd}_0 (\tau) + \overline{f}_0^{\even} (\tau)$.

\begin{definition}\label{5}
For positive integers $r$ and $s$, we define 
\begin{align*}
Z^{\even \odd}_{r,s} (\tau) &= \sum_{\begin{subarray}{c} m>m'>0 \\ m\in\even, m'\in \odd \end{subarray} } \varphi_r (m\tau) \varphi_s (m'\tau)+\frac{1}{4} \varepsilon^{\even \odd}_{r,s} (\tau) , \\
Z^{\odd \even}_{r,s} (\tau) &= \sum_{\begin{subarray}{c} m>m'>0 \\ m\in\odd, m'\in \even \end{subarray} } \varphi_r (m\tau) \varphi_s (m'\tau) + \cz (s) f^{\odd}_r (\tau) +\frac{1}{4} \varepsilon^{\odd \even}_{r,s} (\tau) , \\
 Z^{\odd \odd}_{r,s} (\tau) &=\sum_{\begin{subarray}{c} m>m'>0 \\ m\in\odd, m'\in \odd \end{subarray} } \varphi_r (m\tau) \varphi_s (m'\tau)\\
 &  + \sum_{p+h=r+s} \left( (-1)^s \binom{p-1}{s-1} + (-1)^{p+r} \binom{p-1}{r-1} \right) \cz(p)  f^{\odd}_h (\tau)+ \frac{1}{4}  \varepsilon^{\odd \odd}_{r,s} (\tau)  .
\end{align*}
\end{definition}

\noindent
We can now formulate our double shuffle relations.

\begin{theorem}\label{6}
For positive integers $r$ and $s$, we have
\begin{align}
\label{e_3_4}
Z^{\odd}_r (\tau) Z^{\even}_s(\tau) + \dfrac{1}{4} (\delta_{r,2} \overline{f}_s^{\even} (\tau) + \delta_{s,2} \overline{f}_r^{\odd} (\tau) )& = Z^{\odd \even}_{r,s}(\tau) +Z^{\even \odd}_{s,r}(\tau) \\
\notag & = \sum_{i+j=r+s} \binom{i-1}{r-1} Z^{\odd \even}_{i,j}(\tau) +\sum_{i+j=r+s} \binom{i-1}{s-1} Z^{\odd \odd}_{i,j}(\tau), \\
\label{e_3_5}
Z^{\odd}_r (\tau) Z^{\odd}_s(\tau) + \frac{1}{4} (\delta_{r,2} \overline{f}_s^{\odd} (\tau) + \delta_{s,2} \overline{f}_r^{\odd} (\tau) ) &= Z^{\odd \odd}_{r,s}(\tau)+Z^{\odd \odd}_{s,r}(\tau)+Z^{\odd}_{r+s} (\tau)\\
\notag & = \sum_{i+j=r+s} \left( \binom{i-1}{r-1}+\binom{i-1}{s-1} \right) Z^{\even \odd}_{i,j} (\tau).
\end{align}
\end{theorem}
\begin{proof}
Let $k=r+s$. The proof will be divided into two steps. We first prove the equalities of the imaginary parts in Theorem \ref{6}. The only imaginary parts that appear come from the constant terms $\cz (s)$ of $Z^{\even}_s(\tau)\ (s:{\rm odd})$, $\cz (s)$ in $Z^{\odd \even}_{r,s}(\tau)\ (s:{\rm odd})$ or $\cz(p) \ (p:{\rm odd})$ in $Z^{\odd \odd}_{r,s}(\tau)$. In \eqref{e_3_5}, no imaginary part appears since those of $Z^{\odd \odd}_{r,s}(\tau)+Z^{\odd \odd}_{s,r}(\tau)$ cancels. To prove the equalities of the imaginary parts of \eqref{e_3_4}, we consider the generating functions as follows:
\begin{align*}
Z^{\odd \odd}_k(X,Y)&:=\sum_{r+s=k} {\rm Im\ }Z^{\odd \odd}_{r,s} X^{r-1}Y^{s-1} \\
& = \sum_{r+s=k}\sum_{\begin{subarray}{c} p+h=k \\ p:{\rm odd} \end{subarray} } \left( (-1)^s \binom{p-1}{s-1} + (-1)^{p+r} \binom{p-1}{r-1} \right) \cz(p)  f^{\odd}_h (\tau)X^{r-1}Y^{s-1} \\
&= \sum_{\begin{subarray}{c} p+h=k\\ p:{\rm odd} \end{subarray}} (Y^{h-1}-X^{h-1}) (Y-X)^{p-1} \cz (p) f_h^{\odd} (\tau) ,\\
Z^{\odd \even}_k(X,Y)&:=\sum_{r+s=k} {\rm Im\ }Z^{\odd \even}_{r,s} X^{r-1}Y^{s-1}  = \sum_{\begin{subarray}{c} r+s=k\\ s:{\rm odd} \end{subarray}} \cz(s) f_r^{\odd} (\tau)X^{r-1} Y^{s-1}.
\end{align*}
We note that the imaginary part of the R.H.S. of \eqref{e_3_4} is the coefficient of $X^{r-1} Y^{s-1}$ of $Z^{\odd \even}_k(X+Y,Y)+Z^{\odd \odd}_k(X+Y,X)$. Since we have 
\begin{align*}
&Z^{\odd \even}_k(X+Y,Y)+Z^{\odd \odd}_k(X+Y,X) \\
 =&\sum_{\begin{subarray}{c} r+s=k\\ s:{\rm odd} \end{subarray}} \cz(s) f_r^{\odd} (\tau)(X+Y)^{r-1} Y^{s-1} +  \sum_{\begin{subarray}{c} r+s=k\\ s:{\rm odd} \end{subarray}} (X^{r-1}-(X+Y)^{r-1}) (-Y)^{s-1} \cz (s) f_r^{\odd} (\tau) \\
 = &\sum_{\begin{subarray}{c} r+s=k\\ s:{\rm odd} \end{subarray}} \left( (X+Y)^{r-1} Y^{s-1} + (X^{r-1}-(X+Y)^{r-1}) Y^{s-1} \right)  \cz (s) f_r^{\odd} (\tau)\\
 =&\sum_{\begin{subarray}{c} r+s=k\\ s:{\rm odd} \end{subarray}}   \cz (s) f_r^{\odd} (\tau) X^{r-1} Y^{s-1} ,
\end{align*}
 the assertion follows. Secondly, we prove the equalities of the real parts in Theorem \ref{6}. Again we use generating functions. Define
\begin{align*}
f^{\even \odd}_{r,s} (\tau)& =  \sum_{\begin{subarray}{c} m>m'>0 \\ m\in\even, m'\in \odd \end{subarray} } \varphi_r (m\tau) \varphi_s (m'\tau) , \ f^{\odd \even}_{r,s} (\tau) =  \sum_{\begin{subarray}{c} m>m'>0 \\ m\in\odd, m'\in \even \end{subarray} } \varphi_r (m\tau) \varphi_s (m'\tau) ,\\
f^{\odd \odd}_{r,s} (\tau)& =  \sum_{\begin{subarray}{c} m>m'>0 \\ m\in\odd, m'\in \odd \end{subarray} } \varphi_r (m\tau) \varphi_s (m'\tau) ,\\
\beta_{r,s}^{\odd \odd} (\tau) &=   \sum_{p+h=r+s} \left( (-1)^s \binom{p-1}{s-1} + (-1)^{p+r} \binom{p-1}{r-1} \right) \beta_p  f^{\odd}_h (\tau) ,
\end{align*}
where $\beta_p = - B_p /2p! ( = \cz(p) ,p: {\rm even})$. Consider 
\begin{align}
\notag Z^{\odd}(X) & := \sum_{r\geq1} Z^{\odd}_r X^{r-1} + \alpha_4 \cdot X = \sum_{r\geq1 }f^{\odd}_r (\tau)X^{r-1} + \alpha_4\cdot X,\\
\notag Z^{\even \odd}(X,Y)& := \sum_{r,s\geq1} {\rm Re \ } Z^{\even \odd}_{r,s} X^{r-1} Y^{s-1}= \sum_{r,s\geq1} \Bigl( f^{\even \odd}_{r,s} (\tau)+\frac{1}{4} \varepsilon^{\even \odd}_{r,s} (\tau) \Bigr)  X^{r-1} Y^{s-1}, 
\end{align}
\begin{align}
\notag Z^{\odd \even}(X,Y) &:= \sum_{r,s\geq1} {\rm Re \ } Z^{\odd \even}_{r,s} X^{r-1} Y^{s-1} = \sum_{r,s\geq1}\Bigl(  f^{\odd \even}_{r,s} (\tau)+\beta_s f_r^{\odd} (\tau) +\frac{1}{4} \varepsilon_{r,s}^{\odd \even} (\tau) \Bigr) X^{r-1}Y^{s-1} ,\\
\notag Z^{\odd \odd}(X,Y) &:= \sum_{r,s\geq1} {\rm Re \ } Z^{\odd \odd}_{r,s} X^{r-1} Y^{s-1} = \sum_{r,s\geq1 } \Bigl( f^{\odd \odd}_{r,s} (\tau) +\beta^{\odd \odd}_{r,s} (\tau) +\frac{1}{4} \varepsilon_{r,s}^{\odd \odd} (\tau) \Bigr) X^{r-1} Y^{s-1}, \\
\notag P^{\odd \even}(X,Y) &:=  \sum_{r,s\geq1} {\rm Re \ }\left( Z^{\odd}_r \cdot Z^{\even}_s + \dfrac{1}{4} (\delta_{r,2} \overline{f}_s^{\even} (\tau) + \delta_{s,2} \overline{f}_r^{\odd} (\tau) ) \right) X^{r-1} Y^{s-1} , \\
\label{e_e} P^{\odd \odd}(X,Y) &:=  \sum_{r,s\geq1} {\rm Re \ }\left( Z^{\odd}_r \cdot Z^{\odd}_s + \dfrac{1}{4} (\delta_{r,2} \overline{f}_s^{\odd} (\tau) + \delta_{s,2} \overline{f}_r^{\odd} (\tau) )  \right) X^{r-1}Y^{s-1} ,
\end{align}
where $\alpha_4 =- \alpha_3/2$. Then, it is sufficient to prove that 
\begin{align}
\label{e_3_6}
P^{\odd \even} (X,Y) &= Z^{\odd \even}(X,Y)+Z^{\even \odd} (Y,X) = Z^{\odd \even}(X+Y,Y)+Z^{\odd \odd}(X+Y,X) ,\\
\label{e_3_7}
P^{\odd \odd}(X,Y)&= Z^{\odd \odd}(X,Y)+Z^{\odd \odd}(Y,X)+\frac{Z^{\odd}(X)-Z^{\odd}(Y)}{X-Y} =Z^{\even \odd}(X+Y,Y)+Z^{\even \odd}(X+Y,X).
\end{align}
Now we check the equalities in (20) and (21). Write $\gamma (X)$ and $\gamma (X,Y)$ for the generating functions $\sum_{k\geq1} \gamma_k X^{k-1} $ and $\sum_{r,s\geq1} \gamma_{r,s} X^{r-1}Y^{s-1} $ associated with sequences $\{ \gamma_k \} $ and $\{ \gamma_{r,s} \}$ indexed by one and two integers, respectively. Then we have
\begin{align*}
\beta(X) &= \sum_{k\geq1} \beta_k X^{k-1} =\frac{1}{2} \left( \frac{1}{X} - \frac{1}{e^X-1} \right), \\
f^{\odd} (X) &= \sum_{k\geq1} f^{\odd}_k (\tau) X^{k-1} = - \sum_{u>0} e^{-uX} \frac{q^u}{1-q^{2u}} , \\
f^{\even} (X) &= \sum_{k\geq1} f^{\even}_k(\tau) X^{k-1} = - \sum_{u>0} e^{-uX} \frac{q^{2u}}{1-q^{2u}}, \\
\overline{f}^{\odd} (X)& = \sum_{k\geq1} \overline{f}^{\odd}_k (\tau) X^{k-1} = \frac{1}{X} \left( \sum_{u>0} e^{-uX} \frac{2q^u}{(1-q^{2u})^2} + f^{\odd} (X) - \overline{f}^{\odd}_0 (\tau) \right) ,\\
\overline{f}^{\even} (X)& = \sum_{k\geq1} \overline{f}^{\odd}_k (\tau) X^{k-1} = \frac{1}{X} \left( \sum_{u>0} e^{-uX} \frac{2q^{2u}}{(1-q^{2u})^2} -\overline{f}^{\even}_0 (\tau) \right),\\
f^{\even \odd} (X,Y) &= \sum_{r,s\geq1} f^{\even \odd}_{r,s} (\tau) X^{r-1}Y^{s-1} = \sum_{u,v>0}e^{-uX-vY} \sum_{\begin{subarray}{c} m>m'>0 \\ m\in \even, m' \in \odd \end{subarray} } q^{um+vm'} \\
&= \sum_{u,v>0} e^{-uX-vY} \frac{q^u}{1-q^{2u}} \frac{q^{u+v} }{1-q^{2(u+v)}}, \\
f^{\odd \even} (X,Y) &= \sum_{r,s\geq1} f^{\odd \even}_{r,s} (\tau) X^{r-1}Y^{s-1} = \sum_{u,v>0} e^{-uX-vY} \frac{q^u}{1-q^{2u}} \frac{q^{2(u+v)} }{1-q^{2(u+v)}} ,\\
f^{\odd \odd} (X,Y) &= \sum_{r,s\geq1} f^{\odd \odd}_{r,s} (\tau) X^{r-1}Y^{s-1} = \sum_{u,v>0} e^{-uX-vY} \frac{q^{2u}}{1-q^{2u}} \frac{q^{u+v} }{1-q^{2(u+v)}} ,
\end{align*}
\begin{align*}
\beta^{\odd \odd} (X,Y) &= \sum_{r,s\geq1} \beta_{r,s}^{\odd \odd} (\tau) X^{r-1} Y^{s-1} \\
& = \sum_{p,h\geq1} \beta_p f^{\odd}_h (\tau) \sum_{r+s=p+h}  \left( (-1)^s \binom{p-1}{s-1} + (-1)^{p+r} \binom{p-1}{r-1} \right) X^{r-1} Y^{s-1} \\
&= \sum_{p,h\geq1} \beta_p f^{\odd}_h (\tau)  (X-Y)^{p-1} (Y^{h-1} -X^{h-1} ) = \beta(X-Y) (f^{\odd} (Y) - f^{\odd} (X) ) ,\\
\varepsilon^{\even \odd}(X,Y)& = X\overline{f}^{\odd} (Y) -Y \overline{f}^{\odd} (Y) - \overline{f}^{\odd}_0 (\tau) + X \overline{f}^{\even} (X) + \overline{f}^{\even}_0 (\tau) +\alpha_1 ,\\
\varepsilon^{\odd \even} (X,Y)& = X\overline{f}^{\even} (Y) -Y \overline{f}^{\even} (Y) - \overline{f}^{\even}_0 (\tau) + X \overline{f}^{\odd} (X) + \overline{f}^\odd_0 (\tau) +\alpha_2 ,\\
\varepsilon^{\odd \odd}(X,Y) &= X\overline{f}^{\odd} (Y) -Y \overline{f}^{\odd} (Y) +X \overline{f}^{\odd} (X)+2f^{\odd} (X)+\alpha_3 .
\end{align*} 
By the definitions \eqref{e_e}, we find 
\begin{align*}
Z^{\odd}(X)&= f^{\odd}(X)+\alpha_4 \cdot X ,\\
Z^{\even \odd}(X,Y)&= f^{\even \odd} (X,Y) + \frac{1}{4} \varepsilon^{\even \odd} (X,Y) , \\
Z^{\odd \even}(X,Y) &= f^{\odd \even} (X,Y) + f^{\odd} (X) \beta(Y) +\frac{1}{4} \varepsilon^{\odd \even} (X,Y) ,\\
Z^{\odd \odd}(X,Y) &= f^{\odd \odd} (X,Y) +\beta^{\odd \odd} (X,Y) + \frac{1}{4} \varepsilon^{\odd \odd} (X,Y) ,\\
P^{\odd \even}(X,Y) &= f^{\odd}(X) f^{\even} (Y) +f^{\odd} (X) \beta(Y) +\frac{1}{4} (  X \overline{f}^{\even} (Y) +Y \overline{f}^{\odd} (X) ), \\
P^{\odd \odd}(X,Y)&=  f^{\odd} (X) f^{\odd} (Y) +\frac{1}{4} (X\overline{f}^{\odd} (Y) + Y \overline{f}^{\odd} (X) ).
\end{align*}

\noindent
For the first equality in \eqref{e_3_7}, we compute
\begin{align*}
&f^{\odd \odd} (X,Y)  + f^{\odd \odd} (Y,X) \\
&= \sum_{u,v>0} e^{-uX-vY} \left( \frac{q^{2u}}{1-q^{2u}} + \frac{q^{2v}}{1-q^{2v}} \right) \frac{q^{u+v} }{1-q^{2(u+v)}} \\
&= \sum_{u,v>0} e^{-uX-vY} \left( \frac{q^u}{1-q^{2u}} \frac{q^v}{1-q^{2v}} - \frac{q^{u+v}}{1-q^{2(u+v)}} \right) \\
&= f^{\odd} (X) f^{\odd} (Y) - \sum_{w>u>0} e^{-(w-u)Y-uX} \frac{q^w}{1-q^{2w}} \quad (w=u+v) \\
&= f^{\odd} (X) f^{\odd} (Y) -\sum_{w>0} \frac{q^w}{1-q^{2w}}e^{-wY} \left( e^{Y-X} \frac{1-e^{(Y-X)(w-1)}}{1-e^{Y-X}} \right) \\
&= f^{\odd} (X) f^{\odd} (Y) + \frac{e^Y}{e^X-e^Y} f^{\odd} (Y) - \frac{e^X}{e^X-e^Y} f^{\odd} (X) \\
&=  f^{\odd} (X) f^{\odd} (Y) - \frac{1}{2} (f^{\odd} (X) + f^{\odd} (Y) ) -\frac{1}{2} \coth \left( \frac{X-Y}{2} \right)  (f^{\odd} (X) - f^{\odd} (Y) ),
\end{align*}
\begin{align*}
 \beta^{\odd \odd} (X,Y) +\beta^{\odd \odd}(Y,X) &= (\beta(Y-X)-\beta(X-Y))(f^{\odd} (X) - f^{\odd} (Y)) \\
&= - \frac{f^{\odd} (X)-f^{\odd} (Y) }{X-Y} + \frac{1}{2} \coth \left( \frac{X-Y}{2} \right) (f^{\odd} (X)-f^{\odd} (Y)), \\
\varepsilon^{\odd \odd} (X,Y) + \varepsilon^{\odd \odd}(Y,X) 
&= X\overline{f}^{\odd} (Y) + Y \overline{f}^{\odd} (X) + 2 f^{\odd} (X) +2f^{\odd} (Y) +2 \alpha_3.
\end{align*}
\noindent
Combining these with $(Z^{\odd}(X)-Z^{\odd}(Y))/(X-Y)$, we have
\begin{align*}
&Z^{\odd \odd}(X,Y)+Z^{\odd \odd}(Y,X)+\frac{Z^{\odd}(X)-Z^{\odd}(Y)}{X-Y} \\
&= f^{\odd} (X) f^{\odd} (Y) - \frac{1}{2} (f^{\odd} (X) + f^{\odd} (Y) ) -\frac{1}{2} \coth \left( \frac{X-Y}{2} \right)  (f^{\odd} (X) - f^{\odd} (Y) )\\
&- \frac{f^{\odd} (X)-f^{\odd} (Y) }{X-Y} + \frac{1}{2} \coth \left( \frac{X-Y}{2} \right) (f^{\odd} (X)-f^{\odd} (Y))\\
&+\frac{1}{4} \left(  X\overline{f}^{\odd} (Y) + Y \overline{f}^{\odd} (X) + 2 f^{\odd} (X) +2f^{\odd} (Y) +2 \alpha_3 \right) + \frac{f^{\odd} (X) - f^{\odd} (Y) }{X-Y} + \alpha_4 \\
&= f^{\odd} (X) f^{\odd} (Y) +\frac{1}{4} (X\overline{f}^{\odd} (Y) + Y \overline{f}^{\odd} (X) ) .
\end{align*}

\noindent
For the second equality of \eqref{e_3_7}, we proceed as follows:
\begin{align*}
&f^{\even \odd}(X+Y,X) +f^{\even \odd} (X+Y,Y) \\
&= \sum_{u,v>0}\left(  e^{-(u+v)X-uY} + e^{-uX-(u+v)Y} \right)  \frac{q^{u}}{1-q^{2u}} \frac{q^{u+v}}{1-q^{2(u+v)}} \\
&= \left( \sum_{w>u>0} +\sum_{u>w>0} \right) e^{-uX-wY} \frac{q^{u}}{1-q^{2u}} \frac{q^w}{1-q^{2w}} \\
&= \left( \sum_{w,u>0} -\sum_{w=u>0} \right) e^{-uX-wY} \frac{q^{u}}{1-q^{2u}} \frac{q^w}{1-q^{2w}} \\
&= f^{\odd} (X) f^{\odd} (Y) - \sum_{u>0} e^{-u(X+Y)} \frac{q^{2u}}{(1-q^{2u})^2} \\
&= f^{\odd} (X) f^{\odd} (Y)  - \frac{1}{2} \overline{f}^{\even}_0(\tau) - \frac{1}{2} (X+Y) \overline{f}^{\even} (X+Y),\\
&\varepsilon^{\even \odd} (X+Y,X) + \varepsilon^{\even \odd} (X+Y,Y) \\
&= Y\overline{f}^{\odd} (X) + X\overline{f}^{\odd} (Y) + 2 (X+Y) \overline{f}^{\even} (X+Y) + 2 (\alpha_1- \overline{f}^{\odd}_0 (\tau) +\overline{f}^{\even}_0 (\tau)) .
\end{align*}
\noindent
Summing these up, we have
\begin{align*}
& Z^{\even \odd} (X+Y,Y)+Z^{\even \odd}(X+Y,X)\\
&=f^{\odd} (X) f^{\odd} (Y)  - \frac{1}{2} \overline{f}^{\even}_0(\tau) - \frac{1}{2} (X+Y) \overline{f}^{\even} (X+Y)\\
&+\frac{1}{4} \left( Y\overline{f}^{\odd} (X) + X\overline{f}^{\odd} (Y) + 2 (X+Y) \overline{f}^{\even} (X+Y) + 2 (\alpha_1- \overline{f}^{\odd}_0 (\tau) +\overline{f}^{\even}_0 (\tau)) \right) \\
&= f^{\odd} (X) f^{\odd} (Y) +\frac{1}{4} (X\overline{f}^{\odd} (Y) + Y \overline{f}^{\odd} (X) ) .
\end{align*}


\noindent
For the second equality of \eqref{e_3_6}, we compute
\begin{align*}
& f^{\odd \even} (X,Y) + f^{\even \odd} (Y,X)  \\
&= \sum_{u,v>0} e^{-uX-vY} \left( \frac{q^v}{1-q^{2v}} + \frac{q^{2u+v}}{1-q^{2u}} \right) \frac{q^{u+v} }{1-q^{2(u+v)}} \\
&= \sum_{u,v>0} e^{-uX-vY} \frac{q^u}{1-q^{2u}} \frac{q^{2v}}{1-q^{2v}} = f^{\odd}(X) f^{\even}(Y),\\
&\varepsilon^{\odd \even} (X,Y) + \varepsilon^{\even \odd} (Y,X) = X\overline{f}^{\even} (Y) + Y \overline{f}^{\odd} (X) + \alpha_1+\alpha_2,
\end{align*}
\noindent
to obtain
\begin{align*}
& Z^{\odd \even}(X,Y)+Z^{\even \odd}(Y,X) =  f^{\odd}(X) f^{\even}(Y)  + f^{\odd} (X) \beta (Y) +\frac{1}{4} (X\overline{f}^{\even} (Y) + Y \overline{f}^{\odd} (X) ).
\end{align*}


\noindent
Finally, for the second equality of \eqref{e_3_6}, we similarly compute
\begin{align*}
&f^{\odd \even} (X+Y,Y)+f^{\odd \odd} (X+Y,X) \\
&=\left( \sum_{w>u>0} + \sum_{u>w>0} \right) e^{-uX-vY} \frac{q^u}{1-q^{2u}} \frac{q^{2w} }{1-q^{2w}}\\ 
&=\left( \sum_{w,u>0} - \sum_{u=w>0} \right) e^{-uX-vY} \frac{q^u}{1-q^{2u}} \frac{q^{2w} }{1-q^{2w}} \\
&= f^{\odd}(X) f^{\even} (Y) - \sum_{u>0} e^{-u(X+Y) } \frac{q^{3u}}{(1-q^{2u})^2} \\
&= f^{\odd}(X) f^{\even} (Y) - \sum_{u>0} e^{-u(X+Y) } \left(  \frac{q^{u}}{(1-q^{2u})^2} - \frac{q^{u}}{1-q^{2u}} \right)\\
&= f^{\odd}(X) f^{\even} (Y) -\frac{1}{2} \left( (X+Y) \overline{f}^{\odd} (X+Y) -f^{\odd} (X+Y) +\overline{f}^{\odd}_0(\tau) \right)  -f^{\odd} (X+Y)\\
&= f^{\odd}(X) f^{\even} (Y) -\frac{1}{2}f^{\odd} (X+Y)- \frac{1}{2}(X+Y) \overline{f}^{\odd} (X+Y)- \frac{1}{2} \overline{f}^{\odd}_0(\tau) ,\\
&f^{\odd} (X+Y)\beta(Y) + \beta^{\odd \odd} (X+Y,X) =f^{\odd} (X) \beta(Y),
\end{align*}
\begin{align*}
& \varepsilon^{\odd \even} (X+Y,Y) + \varepsilon^{\odd \odd}(X+Y,X) \\
&= X \overline{f}^{\even} (Y) +Y \overline{f}^{\odd} (X) + 2 (X+Y) \overline{f}^{\odd} (X+Y)+ 2f^{\odd} (X+Y) - \overline{f}^{\even}_0(\tau) +\overline{f}^{\odd}_0 (\tau) +\alpha_3+\alpha_2,
\end{align*}
\noindent
which gives
\begin{align*}
& Z^{\odd \even}(X+Y,Y)+Z^{\odd \odd}(X+Y,X) \\
&= f^{\odd}(X) f^{\even} (Y) -\frac{1}{2}f^{\odd} (X+Y)- \frac{1}{2}(X+Y) \overline{f}^{\odd} (X+Y)- \frac{1}{2} \overline{f}^{\odd}_0(\tau) +f^{\odd} (X) \beta(Y) +\\
&\frac{1}{4} \left(  X \overline{f}^{\even} (Y) +Y \overline{f}^{\odd} (X) + 2 (X+Y) \overline{f}^{\odd} (X+Y)+ 2f^{\odd} (X+Y) - \overline{f}^{\even}_0(\tau) +\overline{f}^{\odd}_0 (\tau) +\alpha_3+\alpha_2 \right) \\
&= f^{\odd}(X) f^{\even} (Y) +f^{\odd} (X) \beta(Y) +\frac{1}{4} (  X \overline{f}^{\even} (Y) +Y \overline{f}^{\odd} (X) ),
\end{align*}
and we are done.

\end{proof}

\section*{Acknowledgements}
The author is grateful to Professor Masanobu Kaneko for initial advice and many useful comments over the course of this work. He also thanks Professor Winfried Kohnen for his information and advice on the references.




\noindent
Koji Tasaka\\
Graduate School of Mathematics, Kyushu University\\
k-tasaka@math.kyushu-u.ac.jp

\end{document}